\documentclass[12pt]{article}
\usepackage{hyperref,amsthm}
\usepackage{amsmath,amssymb,amsfonts,amscd}
\usepackage[T1]{fontenc}
\usepackage[font=footnotesize]{caption}
\usepackage{mathabx,graphicx,verbatim}
\usepackage{color}
\usepackage{tikz}
\usepackage[left=1in,top=1in,right=1in,bottom=1in]{geometry}
\usepackage{refcount}

\renewcommand{\P}{\mathcal{P}}

\newcommand{\F}{\mathcal{F}}
\newcommand{\K}{\mathcal{K}}
\newcommand{\C}{\mathcal{C}}

\newcommand{\G}{\mathcal{G}}
\newcommand{\R}{\mathbb{R}}

\renewcommand{\mod}[1]{\ (\mathrm{mod}\ #1)}

\def\aftermath{\par\vspace{-\belowdisplayskip}\vspace{-\parskip}\vspace{-\baselineskip}}
\def\haftermath{\par\vspace{-\belowdisplayskip}\vspace{-\parskip}\vspace{-.7\baselineskip}}

\newtheorem{thm}{Theorem}
\newtheorem{thmB}{Theorem}

\newtheorem{lem}[thm]{Lemma}

\newtheorem{prop}[thm]{Proposition}
\newtheorem{conj}[thm]{Conjecture}
\newtheorem{dfn}{Definition}

\title{On Asymptotic Packing of Geometric Graphs}
\author{Daniel W. Cranston\footnote{%
Department of Computer Science, Virginia Commonwealth
University, Richmond, VA, USA;
{\tt dcranston@vcu.edu}}
\and
Jiaxi Nie \footnote{Department of Mathematics, UCSD, San Diego, CA, USA; {\tt jin019@ucsd.edu}} \and
Jacques Verstra\"ete \footnote{Department of Mathematics, UCSD, San Diego, CA, USA; {\tt
jacques@ucsd.edu}; Research supported by NSF award DMS-1800332.} \and Alexandra Wesolek \footnote{Department of Mathematics, 
Simon Fraser University, Burnaby, Canada; {\tt agwesole@sfu.ca}; Supported by the Vanier Canada Graduate Scholarships program.}}
\date{\today}

\begin{document}

\maketitle

\begin{abstract}
    A set of geometric graphs is {\em geometric-packable} if it can be asymptotically
packed into every sequence of drawings of the complete graph $K_n$. For example, the set of geometric triangles is geometric-packable due to the existence of Steiner Triple Systems. When $G$ is the $4$-cycle (or $4$-cycle with a chord), we show
that the set of plane drawings of $G$ is geometric-packable. In contrast, the analogous statement is false when $G$ is nearly any other planar Hamiltonian graph (with at most 3 possible exceptions). A
convex geometric graph is {\em convex-packable} if it can be asymptotically
packed into the
convex drawings of the complete graphs.  For each planar Hamiltonian graph $G$, we determine whether or
not a plane $G$ is convex-packable. Many of our proofs explicitly construct these packings; in these cases, the
packings exhibit a symmetry that mirrors the vertex transitivity of $K_n$.
\end{abstract}

\section{Introduction}

A hypergraph $H$ is a pair $(V,E)$ such that $V$ is a set of 
\emph{vertices} and $E$ is a set of subsets of $V$ called \emph{edges}. We
denote by $v(H)$ and $e(H)$ the numbers of vertices and edges in $H$. An
\emph{{$r$}-uniform hypergraph} (\emph{{$r$}-graph} for short) is a hypergraph
where every edge has {$r$} vertices.  When ${r}=2$, this is just a graph.

Let $H$ be a graph and $G$ be a smaller graph. A {\em $G$-packing} of $H$ is a
collection of edge-disjoint subgraphs of $H$ that are isomorphic to $G$. 
We typically seek the maximum cardinality of a $G$-packing of $H$, denoted
by $p(H,G)$. Clearly $p(H,G)\le e(H)/e(G)$. When
equality holds, we have a {\em perfect $G$-packing} of $H$. Let $C_3$ denote a
$3$-cycle, or triangle. Perfect $C_3$-packings of the complete graphs $K_n$ are
called Steiner Triple Systems. It is well known that Steiner Triple Systems
exist if and only if $n\equiv 1$ or $3\mod 6$~\cite{Kirkman1847}. For an arbitrary graph $G$ and any integer $n$ sufficiently large, Wilson~\cite{Wilson1976} gave a necessary and sufficient condition for $K_n$ to have a perfect $G$-packing. In a recent breakthrough, Montgomery, Pokrovskiy,
and Sudakov \cite{montgomery2021proof} proved Ringel's conjecture for sufficiently large $n$, which states that any tree on $n+1$ vertices can be perfectly packed into $K_{2n+1}$. %
Let $\{H_n\}_{n\ge 1}$ be a sequence of graphs such that $v(H_n)=n$. 
Now $H_n$ can be {\em asymptotically packed} by $G$ if
\begin{align}
&\lim_{n\rightarrow\infty}\frac{p(H_n,G)e(G)}{e(H_n)}=1.
\label{asymptotic-packing-def}
\end{align}

Further, $G$ is {\em packable} when $G$ asymptotically packs into the complete
graphs $K_n$.  In other words, $G$ is packable if there
exist $G$-packings of $K_n$ covering all but $o(n^2)$ edges.
Using R\"odl Nibble~\cite{Rodl1985}, one can easily show that all graphs are packable.  See~\cite{Yuster2007} for a survey on graph packing problems.

In this paper, we study 
asymptotic packing problems for geometric graphs, specifically for convex geometric graphs.

\subsection{Geometric Graphs}
A {\em geometric graph} $G$ is a graph
whose vertex set $V(G)$ consists of $n$ points in general position (no three
points on a line) in the Euclidean plane and whose edges are
closed line segments, each with endpoints in $V(G)$.
Two geometric graphs are \emph{congruent} to each other if one can be
transformed into the other by a translation and/or a rotation. A geometric
graph is {\em plane} if no two of its edges cross. An abstract graph is {\em planar} if it has a plane geometric drawing. For any abstract planar graph $G$, let $\P(G)$ denote the set of all plane geometric drawings of $G$. Moreover, let $\P^*(G)$ be the set of all plane geometric drawings of $G$ whose vertices are in strictly convex position; note that $\P^*(G)\subset \P(G)$. For example, if $C_4$ is the $4$-cycle and $G_1$, $G_2$, $G_3$ are the geometric graphs in Figure~\ref{fig:Geometric}, then $G_1\not\in\P(C_4)$ and $G_2\in\P^*(C_4)$ and $G_3\in\P(C_4)\setminus\P^*(C_4)$.

\begin{figure}[h]
    \centering
    \begin{tikzpicture}[thick]
        \tikzstyle{uStyle}=[shape = circle, minimum size = 6.0pt, inner sep = 0pt,
        outer sep = 0pt, draw, fill=white]
        \tikzstyle{lStyle}=[shape = rectangle, minimum size = 20.0pt, inner sep = 0pt,
        outer sep = 2pt, draw=none, fill=none]
        \tikzset{every node/.style=uStyle}
        \begin{scope}[xshift=-1.5in]
        \foreach \i in {1,...,4}
            \draw (45+90*\i:1.2cm) node (v\i) {};
            
            \foreach \i/\j in
            {1/3,2/3,2/4,1/4}
            \draw (v\i) edge (v\j);
            
            \draw (v2) edge[below] node[lStyle]{\footnotesize{$G_1$}} (v3);
        \end{scope}
        
        \begin{scope}[xshift=0in]
        \foreach \i in {1,...,4}
            \draw (45+90*\i:1.2cm) node (v\i) {};
            
            \foreach \i/\j in
            {1/2,2/3,3/4,1/4}
            \draw (v\i) edge (v\j);
            
            \draw (v2) edge[below] node[lStyle]{\footnotesize{$G_2$}} (v3);
        \end{scope}
        
        \begin{scope}[xshift=1.5in,yshift=-.11in, yscale=.935]
         \foreach \i in {0,1,2}
            \draw (90+120*\i:1.2cm) node (v\i) {};
        \draw (0,0) node (v3) {};
        \draw (v0)--(v1)--(v2)--(v3)--(v0);
        
          \draw (v1) edge[below] node[lStyle]{\footnotesize{$G_3$}} (v2);
        \end{scope}

    \end{tikzpicture}
    \caption{$G_1\notin\P(C_4)$ and $G_2\in\P^*(C_4)$ and $G_3\in\P(C_4)\setminus\P^*(C_4)$.\label{fig:Geometric}}
\end{figure}
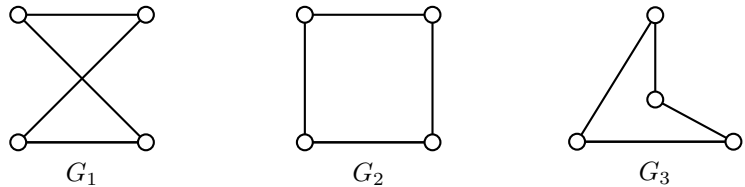

Let $H$ be a geometric graph and let $\G$ be a set of geometric graphs that all have the same number of edges, denoted $e(\G)$. A {\em $\G$-packing}
of $H$ is a collection of edge-disjoint subgraphs of $H$ that are each 
congruent to some member of $\G$. 
Let $p(H,\G)$ be the maximum size of a
$\G$-packing in $H$. Let $\{H_n\}_{n\ge1}$ be a sequence of geometric graphs
such that $v(H_n)=n$. Now $H_n$ can be
\emph{asymptotically packed by $\G$} if 
$$
\lim_{n\rightarrow\infty}\frac{p(H_n,\G)e(\G)}{e(H_n)}=1.
$$

Further, $\G$ is {\em geometric-packable} if any sequence of geometric drawings of complete graphs can be asymptotically packed by $\G$. As stated before, $\P(C_3)$ is geometric-packable, due to the existence of Steiner Triple Systems. To generalize this result, we consider plane geometric packing problems for planar {\em Hamiltonian graphs}, that is, planar graphs that contain a Hamiltonian cycle. We prove the following results.

\begin{thm}\label{thm:GeometricGraphs}
    If $G$ is a planar Hamiltonian graph, then $\P(G)$ is not geometric-packable unless $G$ is the $3$-cycle $C_3$, the $4$-cycle $C_4$, or one of the four graphs $\Theta_1,\Theta_2,\Theta_3,\Theta_4$ shown in Figure~\ref{fig:CyclesWithChords}. Further, $\P(G)$ is geometric-packable if $G$ is one of $C_3$, $C_4$, and $\Theta_1$.
\end{thm}

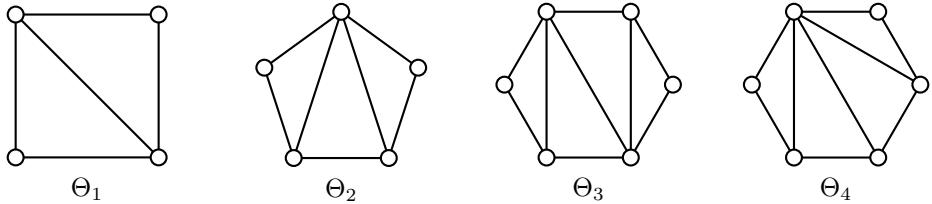
\begin{figure}[h]
    \centering
      \begin{tikzpicture}[thick, scale=.7]
    \tikzstyle{uStyle}=[shape = circle, minimum size = 6.0pt, inner sep = 0pt,
    outer sep = 0pt, draw, fill=white]
    \tikzstyle{lStyle}=[shape = rectangle, minimum size = 20.0pt, inner sep = 0pt,
outer sep = 2pt, draw=none, fill=none]
    \tikzset{every node/.style=uStyle}
    
    \begin{scope}[xshift=-3.6in, yshift=-.01in, rotate=-45, scale=1.2]
            
    \foreach \i in {1,...,4}
    \draw (90*\i:1.6cm) node (v\i) {};
    
    \foreach \i/\j in
    {1/2,2/3,4/1,2/4}
    \draw (v\i) edge (v\j);
    
    \draw (v3) edge[below] node[lStyle]{\footnotesize{$\Theta_1$}} (v4);
    \end{scope}
    
    \begin{scope}[xshift=-1.7in, yshift=-.06in, scale=.96]
        
    \foreach \i in {1,...,5}
    \draw (18+72*\i:1.6cm) node(v\i) {};
    \foreach \i/\j in
    {1/2,2/3,4/5,5/1,1/3,1/4}
    \draw (v\i) edge (v\j);
    
    \draw (v3) edge[below] node[lStyle]{\footnotesize{$\Theta_2$}} (v4);
    \end{scope}
    
    \begin{scope}[xshift=0.15in]
    \foreach \i in {1,...,6}
    \draw (60*\i:1.6cm) node (v\i) {};
    
    \foreach \i/\j in
    {4/3,3/2,2/1,1/5,6/5,1/6,4/2,2/5}
    \draw (v\i) edge (v\j);
    
    \draw (v4) edge[below] node[lStyle]{\footnotesize{$\Theta_3$}} (v5);
    \end{scope}
    
    \begin{scope}[xshift=2.0in]
    \foreach \i in {1,...,6}
    \draw (60*\i:1.6cm) node (v\i) {};
    
    \foreach \i/\j in
    {4/3,3/2,2/1,2/6,6/5,1/6,4/2,2/5}
    \draw (v\i) edge (v\j);
    
    \draw (v4) edge[below] node[lStyle]{\footnotesize{$\Theta_4$}} (v5);
    \end{scope}
    
    \end{tikzpicture}
    \caption{Four plane triangulated cycles.  The first, $\Theta_1$, is geometric-packable.  For each of the remaining three, the question of geometric-packability remains open.}
    \label{fig:CyclesWithChords}
\end{figure}

The plane geometric-packing problem is still open for 3 planar Hamiltonian graphs: $\Theta_2$, $\Theta_3$, and $\Theta_4$. In Theorem~\ref{thm:CyclesWithChords} we show that, in fact, each can be packed when the vertices of the complete geometric graphs are in strictly convex position. We conjecture the following.

\begin{conj}
$\P(G)$ is geometric-packable when $G$ is any of $\Theta_2$, $\Theta_3$, and $\Theta_4$.
\end{conj}

In contrast we show that, if we restrict the geometric graphs that we pack to be both plane and also convex, then every planar Hamiltonian graph, other than the triangle, is not geometric-packable. 

\begin{thm}\label{thm:ConvexGeometricGraphs}
 If $G$ is a planar Hamiltonian graph, then $\P^*(G)$ is not geometric-packable unless $G$ is $C_3$. Further, $\P^*(C_3)$ is geometric-packable.
\end{thm}

The negative part of Theorem~\ref{thm:GeometricGraphs} comes from the fact that, for most planar Hamiltonian graphs $G$, the set $\P(G)$ cannot be packed into the convex drawing of the complete graphs. Note that there is at most one way (in the sense of convex-isomorphism, which we define below) to draw a plane Hamiltonian graph as a plane convex geometric graph. This motivates us to investigate packing problems for convex geometric graphs.

\subsection{Convex Geometric Graphs (CGG)}
A {\em convex geometric graph} (CGG for short) $G$ is a geometric graph whose vertices
are in strictly convex position;  Figure \ref{fig:CyclesWithChords} shows four examples.
We denote the vertices of $G$ by 
$v_0,v_1,\dots,v_{n-1}$ and assume these vertices
appear in clockwise (cyclic) order on the boundary of their convex hull
(indexing is modulo $n$). Any subset of $V(G)$ that is contiguous with respect
to its cyclic order is an {\em interval}.  A pair $\{v_i,v_{i+1}\}$ of $V(G)$
is an {\em extremal pair}. For any pair $\{x, y\}$ in a CGG, the {\em length} of $\{x,y\}$, denoted by $\ell(x,y)$, is the minimum
length of a path from $x$ to $y$ using only extremal pairs. For example, in Figure~\ref{Figure:Example}, $\ell(x,y)=\ell(y,x)=2$.

Informally, two CGGs are convex-isomorphic if some graph isomorphism between them
preserves the cyclic order of all vertices.
Formally, CGGs $G_1$ and $G_2$ are {\em convex-isomorphic}
if there is a bijective function $f: V(G_1)\rightarrow V(G_2)$ such
that for any pair of vertices $x,y\in V(G_1)$: (i) $\ell(x,y)=\ell(f(x),f(y))$; (ii) $\{x,y\}\in E(G_1)$ if and only if $\{f(x),f(y)\}\in E(G_2)$.
A CGG $H$ \emph{contains a CGG $G$} if some subgraph of $H$ is
convex-isomorphic to $G$.

Given CGGs $G$ and $H$, a {\em $G$-packing} of $H$ is a collection of
edge-disjoint subgraphs of $H$ that are convex-isomorphic to $G$. Let $p(H,G)$ be the
maximum size of a $G$-packing in $H$. Let $\mathcal{K}_n$ be the complete CGG
on $n$ vertices. In particular, let $p(n,G)=p(\mathcal{K}_n,G)$. For example,
if $G$ is a $4$-cycle with a crossing, then $p(6,G)=3$.
Figure~\ref{Figure:Example} shows a construction proving the lower bound.  (The
upper bound holds since $\lfloor\binom{6}{2}\slash 4\rfloor=3$.)

\begin{figure}[!h]
    \centering
\begin{tikzpicture}[scale=.6, very thick]
\tikzstyle{uStyle}=[shape = circle, minimum size = 6.0pt, inner sep = 0pt,
outer sep = 0pt, draw, fill=white]
\tikzstyle{lStyle}=[shape = rectangle, minimum size = 20.0pt, inner sep = 0pt,
outer sep = 2pt, draw=none, fill=none]
\tikzset{every node/.style=uStyle}

\foreach \i in {0,...,5}
\draw (60*\i:1in) node (v\i) {};
\draw[red, thick] (v0) -- (v1) -- (v3) -- (v2) -- (v0);
\draw[blue, very thick] (v1) -- (v2) -- (v5) -- (v4) -- (v1);
\draw[green!50!black, ultra thick] (v3) -- (v4) -- (v0) -- (v5) -- (v3);
\draw (180:1.25in) node[lStyle]{$x$};
\draw (60:1.25in) node[lStyle]{$y$};
\end{tikzpicture}
    \caption{When $G$ is a $4$-cycle with a crossing, $p(6,G)=3$.}
    \label{Figure:Example}
\end{figure}
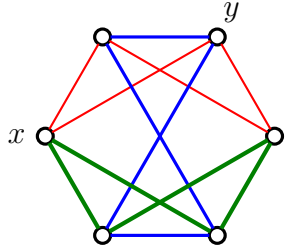

Recall the definition of asymptotically packed from~\eqref{asymptotic-packing-def}.
Further, $G$ is {\em convex-packable} if $G$ asymptotically packs into $\K_n$. That is, $G$ is convex-packable if there exist $G$-packings of
$\K_n$ that cover all but $o(n^2)$ edges.
Finally, a CGG is {\em plane} if no two of its edges cross.

Clearly, for any graph $G$, if $\P(G)$ is geometric-packable and there is only one way to draw $G$ as a plane CGG, then $G$, as a plane CGG, is convex-packable. Let $\C_k$ denote the convex plane $k$-cycle.  Note that $\C_3$ and $\C_4$ are
convex-packable by Theorem~\ref{thm:GeometricGraphs}; so we naturally ask: Is
$\C_5$ convex-packable? The answer is No. In fact, for all $k\ge 5$, the
average length of the edges in a copy of $\C_k$ in $\K_n$ is at most $n/k$;
hence the average length of all edges covered by a $\C_k$-packing of $\K_n$ is
also at most $n/k$. In contrast, the average length of all edges in $\K_n$ is
$(1+o(1))n/4$.  So when $k\ge 5$, no $\C_k$-packing can cover all but $o(n^2)$
edges of $\K_n$. 

By extending this average length argument, we find a necessary
condition (see Lemma~\ref{lem:notpackability}) for a CGG to be convex-packable.
Currently, this is our only tool to prove that a CGG is not convex-packable. But
for many CGGs, it is enough.  We use this argument to prove the
convex-nonpackability of most plane Hamiltonian CGGs.

\begin{thm}
\label{thm:CyclesWithChords}
All plane Hamiltonian CGGs are not convex-packable, except for the two plane cycles $\C_3$ and $\C_4$ and the
four CGGs $\Theta_1,\Theta_2,\Theta_3$ and $\Theta_4$ shown in
Figure~\ref{fig:CyclesWithChords}, which are all convex-packable.
\end{thm}

There is significant research~\cite{AHKVLPSW2017,BK1979,BHRW2006,TCAK2019,biniaz2020packing,obenaus2021edge} about packing large plane trees into geometric graphs. An important conjecture asks whether it is possible to pack $\lfloor \frac{n}{2} \rfloor$ plane spanning trees into a geometric complete graph. This is possible if the drawing of the complete graph is convex~\cite{BK1979,BHRW2006}, but it was very recently shown that the general conjecture is false~\cite{obenaus2021edge}. The current best general construction shows that it is possible to pack $\lfloor \frac{n}{3} \rfloor$ plane spanning trees into any geometric drawing of the complete graph~\cite{biniaz2020packing}. 
In this paper, we relax the requirement that the trees be large, and we focus on packing problems for convex-isomorphic trees in convex geometric complete graphs.  

A plane CGG is a \emph{plane path} if its underlying abstract graph is a path. 
A \emph{convex path} is a plane path in which all edges are extremal.
A \emph{plane convex caterpillar} is a plane CGG such that deleting all leaves
yields a convex path; this resulting path is the caterpillar's \emph{spine}. 

Informally, a reflection of a CGG along an edge ``flips'' all edges and vertices
on one side of that edge.  It is easy to see that each plane path can be
transformed into a plane convex caterpillar by reflections
(Figure~\ref{fig:PathToCaterpillar} shows an example).
Thus, packing problems for plane paths can be reduced
to packing problems for plane convex caterpillars. 

By Lemma~\ref{lem:notpackability}, which formalizes our argument above about
average edge length, every convex path with at least 5 edges is not convex-packable. 
Moreover, Lemma~\ref{lem:notpackability} implies that if a plane path $P$ with
$k$ edges has more than $2\sqrt{k}$ extremal edges, then $P$ is not
convex-packable neither.  And in Section~\ref{sec:Paths}, we prove two partial converses.

The rest of this paper is structured as follows. 
In Section~\ref{sec:notpackable}, we prove a necessary condition for a CGG to be
convex-packable and use it to prove that many CGGs are not convex-packable, which also implies that their corresponding plane geometric graphs are not geometric-packable.
In Section~\ref{sec:geometric}, we prove
Theorems~\ref{thm:GeometricGraphs} and~\ref{thm:ConvexGeometricGraphs}.
In Section~\ref{sec:StronglyPackable},
we define the notion of ``strongly packable'', which strengthens
the property of being convex-packable.
We use this idea to prove Theorem~\ref{thm:CyclesWithChords}, as well as our
results on convex-packable plane paths. 
In Section~\ref{sec:conclude}, we outline directions for future research.

\section{Convex-Nonpackable CGGs}
\label{sec:notpackable}
Sections~\ref{sec:notpackable} and~\ref{sec:StronglyPackable} are mainly focused
on characterizing exactly which CGGs are convex-packable, among all plane
Hamiltonian CGGs.  Most of
these graphs are indeed not convex-packable, as stated in
Theorem~\ref{thm:CyclesWithChords}.  To show this, we develop a general
counting argument to prove convex-nonpackability.  That is the goal of
Section~\ref{sec:notpackability}.
In Section~\ref{sec:notpackability-apps}, we apply this argument to prove the
negative direction of Theorem~\ref{thm:CyclesWithChords}.  (And in
Section~\ref{sec:CyclesWithChords} we prove the positive direction.)

\subsection{A Necessary Condition for Convex-Packability}
\label{sec:notpackability}
Let $G$ be a CGG. A set of edges $\{e_1,e_2,\dots,e_k\}$ of $G$ is {\em convex} if for
each edge
$e_i$, all other edges $e_j$ lie on the same side of the line determined by $e_i$. For example, in Figure~\ref{fig:ConvexSet}, the set of blue edges is convex while the set of red edges is not.

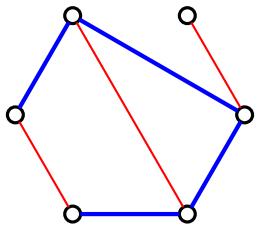
\begin{figure}[h]
    \centering
    \begin{tikzpicture}[scale=.6, very thick]
    \tikzstyle{uStyle}=[shape = circle, minimum size = 6.0pt, inner sep = 0pt,
    outer sep = 0pt, draw, fill=white]
    \tikzstyle{lStyle}=[shape = rectangle, minimum size = 20.0pt, inner sep = 0pt,
    outer sep = 2pt, draw=none, fill=none]
    \tikzset{every node/.style=uStyle}
    
    \foreach \i in {0,...,5}
    \draw (60*\i:1in) node (v\i) {};
    \draw[red, thick] (v0) -- (v1)  (v2) -- (v5)  (v3) -- (v4);
    \draw[blue, ultra thick] (v4) -- (v5) -- (v0) -- (v2) -- (v3);
    
    \end{tikzpicture}
    \caption{The set of thick blue edges is convex, but the set of thin red edges is not.}
    \label{fig:ConvexSet}
\end{figure}

The following lemma is our key tool for proving convex-nonpackability.

\begin{lem}\label{lem:notpackability}
Let $G$ be a CGG.  If $G$ is convex-packable, then there exists $f:E(G)\to \R^+$ that
satisfies the following two conditions: (i) Any set of edges $S$ satisfies
$\sum_{e\in S}f(e)\ge |S|^2/(4e(G))$ and (ii) any convex set of edges $A$
satisfies $\sum_{e\in A}f(e)\le 1$.
\end{lem}
\begin{proof}
For each positive integer $n$, let $\F_n$ be a $G$-packing of $\K_n$ that
covers all but $o(n^2)$ edges. This means that $|\F_n|=(1-o(1))\binom{n}{2}/e(G)$. 
For every $e\in G$ and every copy $G'$ of $G$ in $\K_n$, let $l_{G'}(e)$ denote
the length in $\K_n$ of the edge $F_{G,G'}(e)$.
For any edge $e\in G$, let $f_n(e)$ be the average length of the edges in
$\F_n$ corresponding to edge $e$, divided by $n$; that is,
$$
f_n(e):=\frac{\sum_{G'\in\F_n}l_{G'}(e)}{n|\F_n|}.
$$
Let $S$ be a set of edges of $G$, and denote by $\F_n(S)$ 
the set of edges
in $\K_n$ that are covered in $\F_n$ by edges of $S$.
Note that $|\F_n(S)|=|S||\F_n| =(1-o(1))\frac{|S|}{e(G)}\binom{n}{2}$. 
Further, the average length of edges in $\F_n(S)$ is at least
$$
(1-o(1))\frac{|S|}{4e(G)}n,
$$
since that is precisely the average when $\F_n(S)$ consists of the
$|\F_n(S)|$ shortest edges in $\K_n$, those with length at most
$(1-o(1))\frac{|S|}{e(G)}\frac{n}2$.
Thus, we have
$$
\frac{\sum_{e\in S}f_n(e)}{|S|}\ge(1-o(1))\frac{|S|}{4e(G)},
$$
which we rewrite as
$$
\sum_{e\in S}f_n(e)\ge(1-o(1))\frac{|S|^2}{4e(G)}.
$$
Let $A$ be a convex set of edges of $G$. 
For each $G'\in\F_n$, by definition $\sum_{e\in A}l_{G'}(e)\le n$.
Now summing over all $G'\in \F_n$ (by the definition of $f_n(e)$ above) gives
$$
\sum_{e\in A}f_n(e)=\frac{\sum_{G'\in\F_n}\sum_{e\in A}l_{G'}(e))}{n|\F_n|}\le 1.
$$
Since $\{f_n\}_{n\ge 1}$ is a sequence of bounded functions with finite domain, there 
exists a subsequence $\{f_{n_i}\}_{i\ge1}$ such that $f_{n_i}(e)$ converge for
all $e\in G$. Let $f$ be the function on $E(G)$ such that
$f(e)=\lim_{i\rightarrow\infty}f_{n_i}(e)$. Now $f$ satisfies
(i) and (ii) from the statement of the theorem.
\end{proof}
We can argue that a CGG $G$ is not convex-packable by showing that the function $f$ that would be guaranteed by Lemma~\ref{lem:notpackability} cannot exist. Using an edge length argument similar to the proof of condition (ii), we can further show that a CGG $G$ has no perfect convex-packing\footnote{By  R\"odl Nibble, if a graph has a perfect convex packing into some $\K_n$, then it is convex-packable.}, i.e. $p(n,G)<e(K_n)/e(G)$. Suppose $G$ is the plane $C_4$. The average edge length in every $\K_n$ is strictly greater than $n/4$, but the average edge length in every $C_4$-packing is at most $n/4$.  Therefore, no plane $C_4$-packing covers all edges of $\K_n$. 
\subsection{Applications of Lemma~\ref{lem:notpackability}}
\label{sec:notpackability-apps}
\begin{prop}\label{prop:MaximalConvexSet}
Let $G$ be a CGG and let $k$ be the maximum size of a convex set of edges of
$G$. If $e(G)<k^2/4$, then $G$ is not convex-packable.
\end{prop}

\begin{proof}

We prove the contrapositive.
Assume that $G$ is convex-packable and let $f$ be the function guaranteed
by Lemma~\ref{lem:notpackability}.  Let $A$ be a convex set of edges of size
$k$.  Combining the two conditions in Lemma~\ref{lem:notpackability} gives

$$
1\ge \sum_{e\in A}f(e)\ge \frac{k^2}{4e(G)}.
$$
Simplifying gives $e(G)\ge k^2/4$, as desired.
\end{proof}

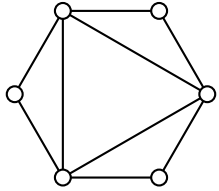
\begin{figure}[!h]
    \centering
\begin{tikzpicture}[thick, scale=.8]
\tikzstyle{uStyle}=[shape = circle, minimum size = 6.0pt, inner sep = 0pt,
outer sep = 0pt, draw, fill=white]
\tikzstyle{lStyle}=[shape = rectangle, minimum size = 6.0pt, inner sep = 0pt,
outer sep = 2pt, draw=none, fill=none]
\tikzset{every node/.style=uStyle}

\foreach \i in {1,...,6}
\draw (60*\i:1.6cm) node (v\i) {};

\foreach \i/\j in
{1/2,2/3,3/4,4/5,5/6,6/1,2/6,2/4,4/6}
\draw (v\i) edge (v\j);

\end{tikzpicture}

    \caption{$\Theta_5$ is not convex-packable.}
    \label{fig:6CycleWith3Chords}
\end{figure}

\begin{prop}\label{prop:Theta5}
If $G$ is a plane Hamiltonian CGG, then $G$
is not convex-packable unless $G$ is one of $\C_3$, $\C_4$, $\Theta_1$,
$\Theta_2$, $\Theta_3$, or $\Theta_4$ (the latter four of these are shown in
Figure~\ref{fig:CyclesWithChords}).
In particular, $\Theta_5$ is not convex-packable.
\end{prop}
\begin{proof}
Let $G$ be a plane CGG formed 
by adding $t$ chords to $\C_k$.  We suppose that $G$ is convex-packable, and apply
Proposition~\ref{prop:MaximalConvexSet}.  When $k=5$, we have
$e(G)\ge\lceil25/4\rceil=7$; when $k=6$, we have $e(G)\ge9$; and when $k\ge
7$, we have $e(G)\ge k^2/4>2k-3$, which contradicts the planarity of $G$.
So the only remaining candidates for a plane convex-packable Hamiltonian CGG are the plane cycles $\C_3$, $\C_4$, the four CGGs in Figure~\ref{fig:CyclesWithChords} and the CGG $\Theta_5$
in Figure~\ref{fig:6CycleWith3Chords}.

Suppose, for contradiction, that $\Theta_5$ is convex-packable
and let $f$ be the function guaranteed by 
Lemma~\ref{lem:notpackability}.
Let $A_1$ be
the convex set consisting of the $6$ edges on the outer 6-cycle and let
$A_2$ be the convex set consisting of the $3$ chords. 
The two conditions in Lemma~\ref{lem:notpackability} give
$$
2 \ge 
\sum_{e\in A_1}f(e)+\sum_{e\in A_2}f(e)=
\sum_{e\in A_1\cup A_2}f(e)\ge \frac{9^2}{4\cdot9}=9/4,
$$
which is a contradiction.
\end{proof}

\section{Proofs of Theorem~\ref{thm:GeometricGraphs} and~\ref{thm:ConvexGeometricGraphs}}
\label{sec:geometric}

\setcounter{thmB}{\getrefnumber{thm:GeometricGraphs}}
\addtocounter{thmB}{-1}

\begin{thmB}
If $G$ is a planar Hamiltonian graph, then $\P(G)$ is not geometric-packable unless $G$ is $C_3,C_4,\Theta_1,\Theta_2,\Theta_3$, or $\Theta_4$ (shown in Figure~\ref{fig:CyclesWithChords}). Further, $\P(G)$ is geometric-packable if $G$ is $C_3$, $C_4$, or $\Theta_1$.
\end{thmB}

\begin{proof}
Note that there is only one way to draw a Hamiltonian graph as a plane CGG. Therefore,
Propositions~\ref{prop:MaximalConvexSet} and~\ref{prop:Theta5} 
imply the geometric-nonpackability of $\P(G)$ for all Hamiltonian graph $G$ unless $G$ is $C_3$, $C_4$, $\Theta_1$, $\Theta_2$,
$\Theta_3$ or $\Theta_4$.
Recall that $\P(C_3)$ is geometric-packable due to the existence of Steiner Triple System. What remains to be done is to prove that $\P(C_4)$ and $\P(\Theta_1)$ are geometric-packable.

We begin with the set of plane $4$-cycles, $\P(C_4)$.
Let $D_n$ be a geometric drawing in the plane of the complete graph $K_n$.
By symmetry, assume that $v_1$ is on the convex hull of the vertices $v_1,\dots,v_n$. 
Assume further that $v_2, \dots, v_n$ appear in clockwise order around $v_1$, with
edges $v_1v_2$ and $v_1v_n$ on the convex hull. The cycles
$v_1,v_i,v_{\lceil n/2 \rceil},v_{\lceil n/2 \rceil+i}$, for each
$i\in\{2,\ldots,\lceil n/2
\rceil-1\}$, are plane and disjoint; so we add them all to our $\P(C_4)$-packing.   Further,
these cycles cover all but $O(1)$ edges incident to $v_1$ or $v_{\lceil
n/2\rceil}$.
Now we delete
$v_1$ and $v_{\lceil n/2 \rceil}$ from the drawing $D_n$ to get a drawing
$D_{n-2}$ of $K_{n-2}$.  We continue recursively, ending when the drawing has
at most 3 vertices.  Each time we delete vertices, we discard $O(1)$ uncovered
incident edges.  Thus, the resulting packing covers all but $O(n)$ edges of
$D_n$.

Now we consider the set of plane $4$-cycles with a chord, $\P(\Theta_1)$.
Let $D_n$ be an arbitrary geometric drawing in the plane of $K_n$.  
Let $f(n)=2n\log_2 n$. We prove by induction on $n$ that there exists a
$\Theta_1$-packing of $D_n$ that covers all but at most $f(n)$ edges. 

Let $m=\lfloor n/4 \rfloor$. By the Ham Sandwich Theorem, there exist two straight
lines partitioning the plane into 4 parts
$P_1,P_2,P_3,P_4$, in clockwise order, where each part contains at least $m$
vertices. Ignoring up to $3$ vertices, we pick $m$ vertices in each part.  We
denote these vertices by $v_{i,j}$, where $i\in\{1,\ldots,4\}$ and
$j\in\{1,\ldots,m\}$. 

Let $D'_n$ be the spanning subgraph of $D_n$ whose edge
set consists of all edges with endpoints in distinct parts, except for those
with one endpoint in each of $P_2$ and $P_4$. Let $\F_n$ be a collection of
copies of plane $\Theta_1$ whose vertex set is $\{v_{1,
j},v_{2,k},v_{3,j+k},v_{4,k}\}$ and whose chord is $\{v_{1,j},v_{3,j+k}\}$, with
$j,k\in\{1,\ldots, m\}$; here each second index is modulo $m$.
It is easy to check that this is a $\Theta_1$-packing of $D'_n$ that covers all
but at most $3n$ edges.

Note that $D_n\setminus D'_n$ consists of three
complete components, one induced by $P_2\cup P_4$, and the others induced by
$P_1$ and $P_3$.  Thus, by induction, there exists a $\Theta_1$-packing of
$D_n$ such that the number of uncovered edges is at most $f(n/2)+2f(n/4)+3n
\leq f(n)$, as desired.
\end{proof}

\setcounter{thmB}{\getrefnumber{thm:ConvexGeometricGraphs}}
\addtocounter{thmB}{-1}

\begin{thmB}
If $G$ is a Hamiltonian graph, then $\P^*(G)$ is not geometric-packable unless $G$ is $C_3$. Further, $\P^*(C_3)$ is geometric-packable.
\end{thmB}
\begin{proof}
\def\rad{.08cm}
\begin{figure}[h]
    \centering
        \begin{tikzpicture}[rotate=120, yscale=-1]
    \foreach \i in {1,...,3} {
    \coordinate (N\i) at (\i*360/3:2.5);
    }
    \foreach \i in {4,...,6} {
    \coordinate (N\i) at (\i*360/3+4:2);
    }
    \foreach \i in {7,...,9} {
    \coordinate (N\i) at (\i*360/3+8:1.5);
    }
    \foreach \i in {10,...,12} {
    \coordinate (N\i) at (\i*360/3+12:1);
    }
    \foreach \i in {1,...,12} {
    \foreach \j in {1,...,12} {
    \draw[gray] (N\i) -- (N\j);
    }
    }
     \foreach \i in {1,...,12} {      
        \fill[] (N\i) circle (\rad);
        \fill[white] (N\i) circle (.7*\rad);
        }
    \draw[red, very thick] (N4)--(N6) -- (N9) -- (N7)--(N4);
     \fill[red] (N4) circle (\rad);
      \fill[red] (N6) circle (\rad);
       \fill[red] (N9) circle (\rad); 
      \fill[red] (N7) circle (\rad);   
    \end{tikzpicture}
    \ \ \ \ \ \ \ \ \ \ \ \ 
        \begin{tikzpicture}[rotate=120, yscale=-1]
    \foreach \i in {1,...,3} {
    \coordinate (N\i) at (\i*360/3:2.5);
    }
    \foreach \i in {4,...,6} {
    \coordinate (N\i) at (\i*360/3+4:2);
    }
    \foreach \i in {7,...,9} {
    \coordinate (N\i) at (\i*360/3+8:1.5);
    }
    \foreach \i in {10,...,12} {
    \coordinate (N\i) at (\i*360/3+12:1);
    }
    \foreach \i in {1,...,12} {
    \foreach \j in {1,...,12} {
    \draw[gray] (N\i) -- (N\j);
    }
    }
     \foreach \i in {1,...,12} {      
        \fill[black] (N\i) circle (\rad);
        \fill[white] (N\i) circle (.7*\rad);
        }
      \draw[red,very thick] (N4)--(N7) -- (N10) -- (N9)--(N4); 
      \fill[red] (N4) circle (\rad);
        \fill[red] (N7) circle (\rad);
          \fill[red] (N10) circle (\rad);
            \fill[red] (N9) circle (\rad);
    \end{tikzpicture}
    \caption{Two convex $C_4$'s in $D_{12}$.}
    \label{fig:GeneralC4}
\end{figure}
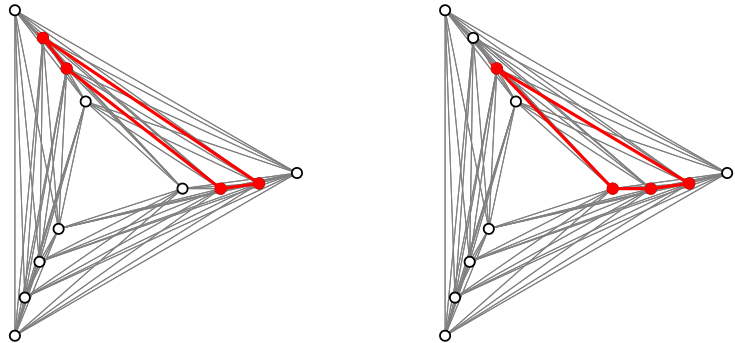

$\P^*(C_3)$ is geometric-packable due to the existence of Steiner Triple Systems. Propositions~\ref{prop:MaximalConvexSet} and~\ref{prop:Theta5} 
imply the geometric-nonpackability of $\P^*(G)$ for all Hamiltonian graph $G$ unless $G$ is one of $C_3$, $C_4$, $\Theta_1$, $\Theta_2$,
$\Theta_3$ and $\Theta_4$. What remains to be done is to construct $D_1,D_2,\dots$, a sequence of drawings of the complete graph, such that it cannot be packed by $P^*(G)$ when $G$ is one of $C_4$, $\Theta_1$, $\Theta_2$,
$\Theta_3$ and $\Theta_4$.

We show that the generalization of the graph in
Figure~\ref{fig:GeneralC4}, which we define below, is not $\P^*(C_4)$-packable.
We only consider drawings $D_n$ when $n \equiv 0 \mod 3$, since we can get
$D_{n-1}$ and $D_{n-2}$ by deleting one or two vertices from $D_n$.

For each $i \in\{1,\ldots,\frac{n}{3}\}$, we let $v_i:=(i,i^2/(2n^2))$. 
The importance of the second coordinate is simply to keep the points
$v_1,\ldots,v_{n/3}$ in general convex position.
We form $v_{i+\frac{n}{3}}$ and $v_{i+\frac{2n}{3}}$ from $v_{i}$
by rotating (around (0,0)) $120$ and $240$ degrees counterclockwise. Let
$A:=\{v_1, \dots, v_\frac{n}{3}\}$, $B:=\{v_{\frac{n}{3}+1}, \dots,
v_\frac{2n}{3}\}$, and $C:=\{v_{\frac{2n}{3}+1}, \dots, v_n\}$. We claim that we
cannot pack many convex $C_4$'s into the drawing on $v_1, \dots,v_n$. The
reason is that each convex $C_4$ has at least two edges within the
sets $A,B,C$. 
To see this, consider a set $S$ with two points from $A$, one point from $B$,
and one point from $C$.
The key observation is that the more inner point from $A$ is not on the convex
hull of $S$.  (This can be verified by finding the line determined by the two points in $A$ and showing that it intersects the interior of the line segment determined by the points in $B$ and $C$.  But we omit these routine calculations.)
Hence, in every convex $C_4$, all four points must come from at most two
of $A$, $B$, and $C$. It is easy to check that the only types of convex $C_4$'s
are the two highlighted in Figure $\ref{fig:GeneralC4}$. But at most
$3\binom{\frac{n}{3}}{2}$ edges go between different sets from $A,B,C$.
Hence, each packing of convex $C_4$'s has at most $\frac{n^2}{3 \cdot
4}(1+o(1))$ copies of $C_4$; so it covers at most $\frac{2}{3}(1+o(1))$
of all edges.  %

By similar arguments, $\{D_n\}_{n\ge1}$ cannot be asymptotically packed by $P^*(G)$ when $G$ is any of $\Theta_1$, $\Theta_2$, $\Theta_3$ and $\Theta_4$.
\end{proof}

\section{Strongly Packable CGGs}
\label{sec:StronglyPackable}

In this section, we introduce strong packability.  This strengthens
(unsurprisingly) the notion of convex-packability, and gives an explicit packing that typically covers all but $O(n)$ edges.  Theorem~\ref{thm:GeometricGraphs} shows that
$\C_4$ is convex-packable.  But to start this section we give a simpler proof, which motivates the machinery we develop and employ in the rest of the section.

For each even integer $n$, 
and all integers $\ell$ such that $1\le l< n/8$, 
we take all copies of $\C_4$ with edges of lengths $2l,n/2-2l,2l-1,n/2-2l+1$,
in clockwise order; see
Figure~\ref{fig:ExplicitlyPackingC4}. Clearly these copies of $\C_4$ form a
$\C_4$-packing of $\K_n$ and the number of edges uncovered is $O(n)$; thus the
packing is almost perfect.

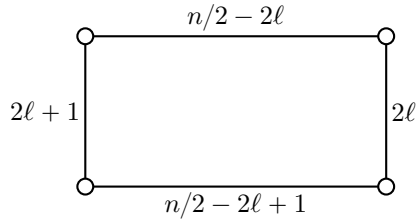
\begin{figure}[h]
    \centering
\begin{tikzpicture}[scale=2,thick]
\tikzstyle{uStyle}=[shape = circle, minimum size = 6.0pt, inner sep = 0pt,
outer sep = 0pt, draw, fill=white]
\tikzstyle{lStyle}=[shape = rectangle, minimum size = 6.0pt, inner sep = 0pt,
outer sep = 2pt, draw=none, fill=white]
\tikzset{every node/.style=uStyle}

\draw (0,0) node (A) {};
\draw (2,0) node (B) {};
\draw (2,-1) node (C) {};
\draw (0,-1) node (D) {};

\draw (A) edge[above] node[lStyle]{\footnotesize{$n/2-2\ell$}} (B);
\draw (B) edge[right] node[lStyle]{\footnotesize{$2\ell$}} (C);
\draw (C) edge[below] node[lStyle]{\footnotesize{$n/2-2\ell+1$}} (D);
\draw (D) edge[left] node[lStyle]{\footnotesize{$2\ell+1$}} (A);

\end{tikzpicture}
    \caption{A typical $\C_4$ in an asymptotic $C_4$-packing of the convex complete graphs.}
    \label{fig:ExplicitlyPackingC4}
\end{figure}

To generalize the example above, that $\C_4$ is convex-packable, we need the
following definitions.
Let $G$ be a subgraph of $\K_n$. The \emph{length set} of $G$, denoted 
$L_{G}$, is a subset of 
$\{1,2,\dots,\lfloor n/2\rfloor\}$
such that $i\in L_{G}$ if and only if some edge $e\in G$ has length
in $\K_n$ equal to $i$.

\begin{dfn}\label{def:length_sequence}
A CGG $G$ is \emph{strongly packable} 
if for each $n\geq 1$ there exists $S_n$, a set of subsets of 
$\{1,2,\dots,\lfloor n/2\rfloor\}$,
with the following three properties:
\begin{enumerate}
    \item All sets in $S_n$ are pairwise disjoint and have size $e(G)$; and
    \item For all $A\in S_n$, there exists a subgraph $G_A$ of $\K_n$ that is
convex-isomorphic to $G$ such that $L_{G_A}=A$; and
    \item $\sum_{A\in S_n}|A|=|S_n|  e(G)=(1-o(1))  n/2$.
\end{enumerate}
Further, we say $\{S_n\}_{n\ge 1}$ is a sequence of \emph{packable length-set collections of $G$}.
\end{dfn}
\subsection{Strong Packability and Reflection}
In this subsection, we prove two easy lemmas about strong packability.
First we show that every strongly packable CGG is also convex-packable.
\begin{lem}\label{lem:StronglyPackable}
Let $G$ be a CGG. If $G$ is strongly packable, then $G$ is also convex-packable.
\end{lem}

\begin{proof}
Let $\{S_n\}_{n\ge 1}$ be a 
sequence of collections of sets that 
satisfy all three properties in Definition~\ref{def:length_sequence}, i.e.,
a sequence of packable length-set collections of $G$.  We construct a sequence of
$G$-packings $\{\F_n\}_{n\ge 1}$ that covers all but $o(n^2)$ edges of
$\{\K_n\}_{n\ge 1}$.
Fix an arbitrary positive
integer $n'$.  For each $A\in S_{n'}$, we add to $\F_{n'}$ the subgraph $G_A$ convex-isomorphic to $G$
in $\K_{n'}$ with lengths in $A$, as well as the $n-1$ copies of $G$ formed from
$G_A$ by rotations.

Each edge $e$ in $\K_{n'}$ of some length $l$, where $l\in\{1,\ldots,\lfloor
n'/2\rfloor\}$,
is covered by $\F_{n'}$ if and only if $l\in \bigcup_{A\in S_{n'}}A$.  Since
$\{S_n\}_{n\ge 1}$  is a sequence of packable length-set collections of $G$, Property~3 of
Definition~\ref{def:length_sequence} implies that the number of lengths not covered
by $\bigcup_{A\in S_n}A$ is $o(n)$.  Thus, the number of edges not covered
by $\{\F_n\}_{n\ge 1}$ is $o(n^2)$, as desired.
\end{proof}
Consider a CGG $G$ with vertex set $\{v_0,\ldots,v_{m-1}\}$. A CGG $\tilde{G}$ is a \emph{reflection}
 of $G$ along an edge $\{v_0,v_t\}$ if $\tilde{G}$ has the same vertex set
 as $G$ and a pair of vertices $\{v_i,v_j\}$ form an edge of $\tilde{G}$ if and only if
either (a) $\{i,j\}\in\{0,\ldots,t\}$ and $\{v_{t-i},v_{t-j}\}\in E(G)$ or
(b) $\{i,j\}\in\{0,\ldots,m-1\}\setminus\{1,\ldots,t-1\}$ and
 $\{v_{i},v_{j}\}\in E(G)$.  Reflections are illustrated in Figure~\ref{fig:PathToCaterpillar} (in Section~\ref{sec:Paths}).
\begin{lem}\label{lem:reflection}
If a CGG $G$ is strongly packable, then so is each of its reflections.
\end{lem}

\begin{proof}
Let $\tilde{G}$ be a reflection of $G$ and let $\{S_n\}_{n\ge 1}$ be a sequence of
collections of packable length-sets of $G$.  Fix an arbitrary positive integer
$n'$.  For each $A\in S_{n'}$, let $H$ be a copy of $G$ in $\K_{n'}$ such that
$L_H=A$.  Clearly $\K_{n'}$ also contains a copy $\tilde{H}$ of $\tilde{G}$
such that $L_{\tilde{H}}=L_H=A$ (which is the reflection of $H$ along the image of the edge of the reflection from $G$ to $\tilde{G}$). Therefore,
$\{S_n\}_{n\ge 1}$ is also a sequence of collections of packable length-sets of
$\tilde{G}$. Hence $\tilde{G}$ is also strongly packable.
\end{proof}

\subsection{Proof of Theorem~\ref{thm:CyclesWithChords}}
\label{sec:CyclesWithChords}
In this section we give more examples of strong packability.
Typically, in all copies of $G$ we assign each given edge $f$ lengths with the
same residue modulo $e(G)$.  This approach allows us to coordinate lengths of
various edges to ensure that the edges form the desired cycles in all copies of $G$.
Here we prove Theorem~\ref{thm:CyclesWithChords}; for convenience, we restate
it.


\setcounter{thmB}{\getrefnumber{thm:CyclesWithChords}}
\addtocounter{thmB}{-1}

\begin{thmB}
All plane Hamiltonian CGGs are not convex-packable, except for the two plane cycles $\C_3$ and $\C_4$ and the
four CGGs $\Theta_1,\Theta_2,\Theta_3$ and $\Theta_4$ shown in
Figure~\ref{fig:CyclesWithChords}, which are all convex-packable.
\end{thmB}
\begin{proof}
Propositions~\ref{prop:MaximalConvexSet} and~\ref{prop:Theta5} 
imply the packability of all plane Hamiltonian CGGs, except for $\C_3$, $\C_4$, $\Theta_1$, $\Theta_2$,
$\Theta_3$ and $\Theta_4$.
Recall that $\C_3$ is convex-packable due to the existence of Steiner Triple System, and that $\C_4$ and
$\Theta_1$ are convex-packable by Theorem~\ref{thm:GeometricGraphs}.

We now consider $\Theta_2$, $\Theta_3$, and $\Theta_4$. We will show
that they are all (strongly) packable.
To show that a CGG is strongly packable, it suffices to give a sequence of
packable length-set collections.   We label the edges of $\Theta_2$ as
in Figure~\ref{fig:Theta2WL}.
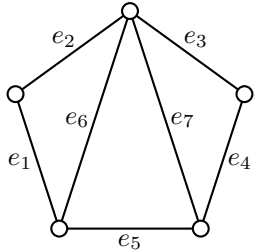
\begin{figure}[h]
    \centering
\begin{tikzpicture}[thick]
\tikzstyle{uStyle}=[shape = circle, minimum size = 6.0pt, inner sep = 0pt,
outer sep = 0pt, draw, fill=white]
\tikzstyle{lStyle}=[shape = rectangle, minimum size = 6.0pt, inner sep = 0pt,
outer sep = 2pt, draw=none, fill=none]
\tikzset{every node/.style=uStyle}

\foreach \i in {1,...,5}
\draw (-54+72*\i:1.6cm) node (v\i) {};

\foreach \i/\j/\where/\num in
{4/3/left/e_1,3/2/above/e_2~~,2/1/above/~~e_3,1/5/right/e_4,5/4/below/e_5,4/2/left/e_6,2/5/right/e_7}
\draw (v\i) edge[\where] node[lStyle]{\footnotesize{$\num$}} (v\j);

\end{tikzpicture}
    \caption{An edge labeling of $\Theta_2$ used to show it is strongly packable.}
    \label{fig:Theta2WL}
\end{figure}

Let $n=28x$, for some positive integer $x$.  We construct a collection of
length-sets for $\Theta_2$ in $\K_n$.  If $28\nmid n$, then we
simply ignore the edges incident to at most 27 vertices in $\K_n$. Since the
number of ignored edges is $O(n)$, they do not affect whether or not our packing covers
$(1-o(1))\binom{n}{2}$ edges.  Thus, we assume $n=28x$ without loss of
generality.

We extend the idea of a length-set to a
\emph{length-vector}.  Each length-vector for $\Theta_2$ is an ordered 7-tuple
of positive integers that form a length-set.  Here the $j$th entry of a
length-vector is the length in some copy of $\Theta_2$ of the image of edge
$e_j$, as labeled in Figure~\ref{fig:Theta2WL}.
We use the following $2x-3$ length-vectors:
$$
\begin{aligned}
(7x+1-7i, 2+14i, 5+14i, 7x+6-7i, 14x-14i-14,\\
7x+3+7i, 7x+11+7i),\ 0\le i \le x-2
\end{aligned}
$$
and
$$
\begin{aligned}
(-4+7i, 14x+12-14i, 14x+9-14i, 4+7i, 14i-21,\\
14x+8-7i, 14x+13-7i),\ 2\le i \le x-1.
\end{aligned}
$$
As in the proof of Lemma~\ref{lem:StronglyPackable}, for each length-vector above, we add to
$\F_n$ some copy of $G$ with edges lengths given by that length-vector, as well
as the $n-1$ non-trivial rotations of that copy of $G$.
%
Note that the 7 coordinates of each length-vector
lie in distinct residue classes modulo $7$. This observation makes it 
easy to check that the length-sets of these $2x-3$ length-vectors 
are pairwise disjoint. Thus the number of edge lengths in these length-vectors
is $7(2x-3)=(1-o(1))n/2$, and the number of edges covered by this
$\Theta_2$-packing is $(1-o(1))\binom{n}{2}$.
So $\Theta_2$ is strongly packable.

The proofs that $\Theta_3$ and $\Theta_4$ are strongly packable mirror that
above for $\Theta_2$.
So we just give the length-vectors (which are identical for $\Theta_3$ and
$\Theta_4$) and the edge labelings in Figure~\ref{fig:Theta3WL}.

\begin{figure}[h]
    \centering
\begin{tikzpicture}[thick, scale=.9]
\tikzstyle{uStyle}=[shape = circle, minimum size = 6.0pt, inner sep = 0pt,
outer sep = 0pt, draw, fill=white]
\tikzstyle{lStyle}=[shape = rectangle, minimum size = 6.0pt, inner sep = 0pt,
outer sep = 2pt, draw=none, fill=none]
\tikzset{every node/.style=uStyle}

\foreach \i in {1,...,6}
\draw (60*\i:1.6cm) node (v\i) {};

\foreach \i/\j/\where/\num in
{4/3/below/e_1~~,3/2/above/e_2~~~,2/1/above/e_6,1/5/right/e_7,5/4/below/e_4,6/5/below/~~~e_9/,1/6/above/~~~e_8,4/2/left/e_3,2/5/left/e_5}
\draw (v\i) edge[\where] node[lStyle]{\footnotesize{$\num$}} (v\j);
\draw (0,-2.5) node[draw=none] {~$\Theta_3$};

\begin{scope}[xshift=1.7in]

\foreach \i in {1,...,6}
\draw (60*\i:1.6cm) node (v\i) {};

\foreach \i/\j/\where/\num in
{4/3/below/e_1~~,3/2/above/e_2~~~,2/1/above/e_9,2/6/below/e_7~~,5/4/below/e_4,6/5/below/~~~e_6/,1/6/above/~~~e_8,4/2/left/e_3,2/5/left/e_5}
\draw (v\i) edge[\where] node[lStyle]{\footnotesize{$\num$}} (v\j);
\draw (0,-2.5) node[draw=none] {~$\Theta_4$};
\end{scope}

\begin{scope}[xshift=4.0in]
\begin{scope}[rotate=-45]
\foreach \i in {1,...,4}
\draw (90*\i:1.95cm) node (v\i) {};

\foreach \i/\j/\where/\num in
{1/2/above/e_4, 2/3/left/e_2, 3/4/below/e_1, 4/1/right/e_5, 2/4/left/e_3}
\draw (v\i) edge[\where] node[lStyle]{\footnotesize{$\num$}} (v\j);
\end{scope}
\draw (0,-2.5) node[draw=none] {~$\Theta_1$};
\end{scope}

\end{tikzpicture}

    \caption{Edge labelings showing that $\Theta_3$, $\Theta_4$, and
$\Theta_1$ are all strongly packable.}
    \label{fig:ThetaWL}
    \label{fig:Theta3WL}
\end{figure}
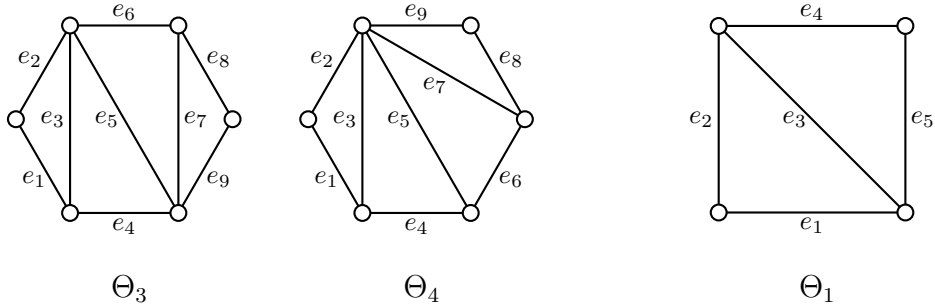

$$
\begin{aligned}
(2x-2i,5x+i,7x-i,2x+2i,9x-i,2x-1-2i,\\
7x+i+1,2x+1+2i, 5x-i),\ 1\le i \le x-1.
\end{aligned}
$$
\haftermath
\end{proof}
In fact, we can also show that $\Theta_1$ is strongly packable using the
edge labeling in Figure~\ref{fig:ThetaWL} and the following collections of
vectors (we assume, without loss of generality, that $n=20x$): 
$$(5i+2,5i+3, 10i+5, 10x-5i-1,10x-5i-4),\ 0\le i< x-1,$$ 
$$(5i+2,5i-2, 20x-10i, 10x-5i-1,10x-5i+1),\ x< i< 2x-1.$$

%

\subsection{Paths and Caterpillars}\label{sec:Paths}
Note that any plane path can be uniquely transformed into a convex plane
caterpillar by reflections (See Figure~\ref{fig:PathToCaterpillar}, where in
each step we reflect with respect to the bold red edge).  Thus, by
Lemma~\ref{lem:reflection}, a plane path is strongly packable if and only if
its corresponding convex plane caterpillar is strongly packable.

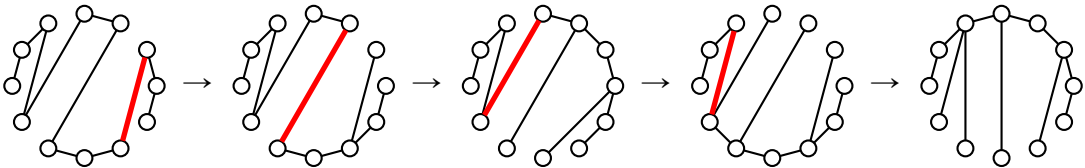
\begin{figure}[h]
    \centering
    \begin{tikzpicture}[thick, scale=.6]
    \tikzstyle{uStyle}=[shape = circle, minimum size = 6.0pt, inner sep = 0pt,
    outer sep = 0pt, draw, fill=white]
    \tikzstyle{lStyle}=[shape = rectangle, minimum size = 20.0pt, inner sep = 0pt,
outer sep = 2pt, draw=none, fill=none]
    \tikzset{every node/.style=uStyle}
    
    \foreach \i in {1,...,12}
     \draw (30*\i:1.6cm) node (v\i) {};
    
    \draw (v6)--(v5)--(v4)--(v7)--(v3)--(v2)--(v8)--(v9)--(v10) (v1)--(v12)--(v11);
    \draw[red, line width=2pt] (v1)--(v10);
    \draw (0: 2.5cm) node[lStyle]{$\rightarrow$};
    
    \begin{scope}[xshift=2.0in]
    \foreach \i in {1,...,12}
     \draw (30*\i:1.6cm) node (v\i) {};
    
    \draw (v6)--(v5)--(v4)--(v7)--(v3)--(v2) (v8)--(v9)--(v10); 
    \draw (v1)--(v10)--(v11)--(v12);
    \draw[red, line width=2pt] (v2)--(v8);
    \draw (0: 2.5cm) node[lStyle]{$\rightarrow$};
    \end{scope}
    
    \begin{scope}[xshift=4.0in]
    \foreach \i in {1,...,12}
     \draw (30*\i:1.6cm) node (v\i) {};
    
    \draw (v6)--(v5)--(v4)--(v7)  (v3)--(v2); 
    \draw (v8)--(v2)--(v1)--(v12); 
    \draw (v9)--(v12)--(v11)--(v10);
    \draw[red, line width=2pt] (v3)--(v7);
    \draw (0: 2.5cm) node[lStyle]{$\rightarrow$};
    \end{scope}
    
     \begin{scope}[xshift=6.0in]
    \foreach \i in {1,...,12}
     \draw (30*\i:1.6cm) node (v\i) {};
    
    \draw (v6)--(v5)--(v4) (v3)--(v7)--(v8); 
    \draw (v2)--(v8)--(v9)--(v10); 
    \draw (v1)--(v10)--(v11)--(v12);
    \draw[red, line width=2pt] (v4)--(v7);
    \draw (0: 2.5cm) node[lStyle]{$\rightarrow$};
    \end{scope}
    
    \begin{scope}[xshift=8.0in]
    \foreach \i in {1,...,12}
    \draw (30*\i:1.6cm) node (v\i) {};
    
    \draw (v6)--(v5)--(v4)--(v7) (v8)--(v4)--(v3); 
    \draw (v9)--(v3)--(v2)--(v1); 
    \draw (v10)--(v1)--(v12)--(v11);
    \end{scope}
    
    \end{tikzpicture}
    \caption{A plane path is transformed into a convex plane caterpillar.}
    \label{fig:PathToCaterpillar}
\end{figure}

By Proposition~\ref{prop:MaximalConvexSet}, we know that a plane path with $k$
edges is not convex-packable if it has more than $2\sqrt{k}$ extremal edges.
Equivalently, a plane convex caterpillar is not convex-packable if it has more than
$2\sqrt{k}-2$ edges in the spine. Let $f(k)$ be the maximum number of extremal edges of a convex-packable plane path with $k$ edges. So we have $f(k)\le 2\sqrt{k}$ for all positive integers $k$. We conjecture that this is sharp.
\begin{conj}
$f(k)=(2-o(1))\sqrt{k}$ as $k\rightarrow\infty$.
\end{conj}
In the rest of this section, we construct two types of strongly packable convex plane caterpillars, both show $f(k)\ge (1+o(1))\sqrt{2k}$ as $k\rightarrow\infty$. 

\begin{thm}\label{thm:CaterpillarWithLegsAtEndpoints}
Let $s$ be a positive integer. Let $G$ be a plane convex caterpillar with
vertices $v_1, \dots, v_{s}$, $v_{s+1}$ in clockwise order on the spine
such that each of $v_2,\ldots,v_{s}$ has no adjacent leaf. If
$s(s+3)\le 2e(G)$, then $G$ is strongly packable.

\end{thm}

If we let $s(s+3)=2e(G)$, then we have a plane convex caterpillar with $s(s+3)/2$ edges among which $s+2$ edges are extremal. This implies that $f(k)\ge (1+o(1))\sqrt{2k}$. 
\begin{proof} 
Let $k:=e(G)$. Without loss of generality, assume $n$ is a multiple of $2k$
and let $x:={n}/(2k)$. We construct copies $G_m$ of $G$ in $\K_n$ for each $m\in
\{1,\ldots, x-1\}$ such that $\{L_{G_m}\}$ is a packable length-set collection of $G$.
Let $t_1$ and $t_2$ denote the numbers of leaves adjacent to $v_1$ and
$v_{s+1}$; by symmetry, we assume $t_2\ge t_1$.  Denote the leaf edges
incident to $v_1$ (resp. to $v_{s+1}$) by $e_{s+1},e_{s+2}, \dots,
e_{s+t_1}$ (resp. $e_{s+t_1+1}, e_{s+t_1+2}, \dots, e_{s+t_1+t_2}$).

 \begin{figure}[h]
    \centering
    \begin{tikzpicture}[thick, scale=.6]
    \tikzstyle{uStyle}=[shape = circle, minimum size = 6.0pt, inner sep = 0pt,
    outer sep = 0pt, draw, fill=white]
    \tikzstyle{lStyle}=[shape = rectangle, minimum size = 20.0pt, inner sep = 0pt,
outer sep = 2pt, draw=none, fill=none]
    \tikzset{every node/.style=uStyle}
        \foreach \i in {1,...,6}
        \draw (160-20*\i:4.0cm) node(v\i) {}; 
        \foreach \i in {1,...,6}
        \draw (160-20*\i:4.5cm) node[lStyle] {\footnotesize{$v_{\i}$}};
        \foreach \i in {1,...,5}
        \draw (150-20*\i:3.5cm) node[lStyle] {\footnotesize{$e_{\i}$}};

        \draw (v1)--(v2)--(v3)--(v4)--(v5)--(v6);
        
        \foreach \i in {7,8,9}
        \draw (45+20*\i:4.0cm) node(v\i) {}; 
        \foreach \i in {6,7,8}
        \draw (65+20*\i:4.7cm) node[lStyle] {\footnotesize{$e_{\i}$}};
        \foreach \i in {7,8,9}
        \draw (v1)--(v\i);
        
        \foreach \i in {10,...,13}
        \draw (200-20*\i:4.0cm) node(v\i) {}; 
        \foreach \i in {9,...,12}
        \draw (180-20*\i:4.7cm) node[lStyle] {\footnotesize{$e_{\i}$}};
        \foreach \i in {10,...,13}
        \draw (v6)--(v\i);
    \end{tikzpicture}
    \caption{An edge-labeling of a caterpillar where all leaves are adjacent to
either $v_1$ or $v_{s+1}$.}
    \label{fig:CaterpillarWithLegsAtEndpoints}
\end{figure}
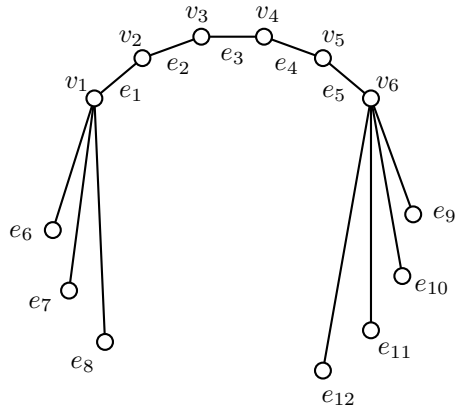
 
Let $G_m$ be a copy of $G$ in $\K_n$, and let
$\ell_m(e):=\ell(F_{G,G_m}(e))$, i.e., $\ell_m(e)$ is the length in $\K_n$ of the
edge of $G_m$ corresponding to the edge $e$ of $G$. 
We will show, for each $m\in\{1,\ldots,x-1\}$, that there exists $G_m$ with
$$
\ell_m(e_i)=
\left\{\!\!\!
\begin{array}{lll}
(i-1)x+m& &~~~~~i\in\{1,\ldots,s\}\\
s x+(m-1)t_1+(i-s)& &~~~~~i\in\{s+1,\ldots, s+t_1\}\\
n/2-(m+1)t_2+(i-s-t_1)& &~~~~~i\in\{s+t_1+1,\ldots, s+t_1+t_2\}.
\end{array}
\right.
$$
Let $S_n$ be the collection of sets $L_{G_m}$ for all $m\in \{1,\ldots,x\}$.
It is straightforward to check that these sets are pairwise disjoint, and
thus that $S_n$ is a packable length-set collection.  
In particular, spine edges all have lengths in $\{1,\ldots,s
x\}$; leaf edges incident to $v_1$ have lengths in $\{s x+1,\ldots,s
x+t_1x\}$, and leaf edges incident to $v_{s+1}$ have lengths in $\{s x+t_1
x+1,\ldots,n/2\}$.  Since $\sum_{m=1}^{x-1} |L_{G_m}| = (1-o(1))n/2$, we only need
to check that $G$ contains the desired copies $G_m$, with no edges
crossing.  It suffices to verify that the edges $e_{s+t_1}$ and
$e_{s+t_1+t_2}$ in each copy do not cross.  
Recall that $t_2\ge t_1$.
Thus, for each
$m\in\{1,\ldots,x-1\}$, because 
\begin{align*}
&\sum_{i=1}^{s}\ell_m(e_i)+\ell_{m}(e_{s+t_1})+\ell_{m}(e_{s+t_1+t_2})\\
=&\sum_{i=1}^s ((i-1)x+m)+s x+mt_1+n/2-(m+1)t_2+t_2\\
=&\frac{s(s-1)x}2+sm+sx+mt_1+n/2-mt_2\\
<&\frac{s(s+3)x}2+n/2 \le k(n/2k)+n/2 =~n.
\end{align*}
\aftermath

\end{proof}

\begin{thm}\label{thm:Caterpillar}
Let $l$, and $a_1, \dots$,$a_{s+1}$ be positive integers. Let $G$ be a plane
convex caterpillar with vertices $v_1, \dots, v_{s}$, $v_{s+1}$ in
clockwise order on the spine such that each $v_i$ is adjacent to $a_i$ leaves.
If there exists a permutation $\sigma_1, \dots,
\sigma_{s-1},\sigma_s$ of the set $1,\dots, s$ such that $\sigma_i\le \min\{a_i, a_{i+1}\}$
for all $i\in\{1,\ldots,s\}$, then $G$ is strongly packable.
\end{thm}
If we let $a_i=i$ and $\sigma_i=i$, then we have a packable plane convex caterpillar with $(s^2+5s+3)/2$ edges among which $s+2$ are extremal. This implies, again, $f(k)\ge (1+o(1))\sqrt{2k}$.

 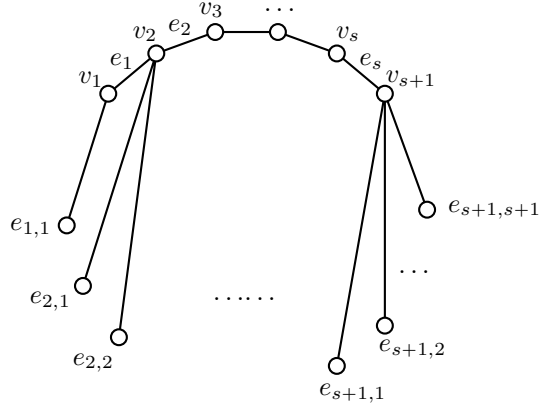
\begin{figure}[h]
    \centering
    \begin{tikzpicture}[thick, scale=.6]
    \tikzstyle{uStyle}=[shape = circle, minimum size = 6.0pt, inner sep = 0pt,
    outer sep = 0pt, draw, fill=white]
    \tikzstyle{lStyle}=[shape = rectangle, minimum size = 20.0pt, inner sep = 0pt,
outer sep = 2pt, draw=none, fill=none]
    \tikzset{every node/.style=uStyle}
        \foreach \i in {1,...,6}
        \draw (160-20*\i:4.0cm) node(v\i) {}; 
        \foreach \i in {1,2,3}
        \draw (160-20*\i:4.5cm) node[lStyle] {\footnotesize{$v_{\i}$}};
        \foreach \i in {1,2}
        \draw (150-20*\i:4.3cm) node[lStyle] {\footnotesize{$e_{\i}$}};
        
        \draw (60:4.5cm) node[lStyle] {\footnotesize{$v_{s}$}};
        \draw (40:4.5cm) node[lStyle] {\footnotesize{$~~v_{s+1}$}};
        
        \draw (80:4.5cm) node[lStyle] {\footnotesize{$\dots$}};
        \draw (50:4.3cm) node[lStyle] {\footnotesize{$e_s$}};
        
        \draw (v1)--(v2)--(v3)--(v4)--(v5)--(v6);
        
        \foreach \i in {7,8,9}
        \draw (45+20*\i:4.0cm) node(v\i) {}; 
        
        \draw (185:4.8cm) node[lStyle] {\footnotesize{$e_{1,1}$}};
        \draw (205:4.8cm) node[lStyle] {\footnotesize{$e_{2,1}$}};
        \draw (225:4.8cm) node[lStyle] {\footnotesize{$e_{2,2}$}};
        \draw (260:2cm) node[lStyle] {\footnotesize{$\dots$}};
        \draw (280:2cm) node[lStyle] {\footnotesize{$\dots$}};

        \draw (v1)--(v7) (v8)--(v2)--(v9);
        
        \foreach \i in {10,12,13}
        \draw (200-20*\i:4.0cm) node(v\i) {}; 
        
        \draw (0:5.5cm) node[lStyle] {\footnotesize{$e_{s+1,s+1}$}};
        \draw (-20:4cm) node[lStyle] {\footnotesize{$\dots$}};
        \draw (-40:4.8cm) node[lStyle] {\footnotesize{$e_{s+1,2}$}};
        \draw (-60:4.7cm) node[lStyle] {\footnotesize{$e_{s+1,1}$}};
        
        \foreach \i in {10,12,13}
        \draw (v6)--(v\i);
    \end{tikzpicture}
    \caption{An edge-labeling of a caterpillar where $v_i$ is adjacent to $i$ leaves for all $1\le i\le s+1$.}
    \label{fig:CaterpillarWithLegsAtEndpointsTwo}
\end{figure}

The proof of Theorem~\ref{thm:Caterpillar} is similar to that of
Theorem~\ref{thm:CaterpillarWithLegsAtEndpoints}.  The main difference is that
now we must ensure that the final leaf edge incident with $v_i$ does not cross
with the first leaf edge incident with $v_{i+1}$, for each $i\in\{1,\ldots,s+1\}$.
That is the role of the permutation $\sigma$, which controls the length of each
edge on the spine.  More precisely, letting $x:=n/(2k)$, in each copy of $G$ the
edge $e_i$ has length between $(\sigma_i-1)x$ and $\sigma_i x-1$.

\begin{proof} 
Let $k:=e(G)$. Without loss of generality, assume $n$ is a multiple of $2k$ and
let $x:={n}/(2k)$.  For each $m\in \{1,\ldots,x-k^2\}$, we will construct a copy
$G_m$ of $G$ in $\K_n$ such that $\{L_{G_m}\}$ is a packable length-set
collection of $G$. Let $e_1,\dots,e_s$ be the edges on the spine, where
$e_{i}=\{v_i,v_{i+1}\}$, and let $e_{i,1}, \dots,e_{i,a_i}$ be the
edges connecting $v_i$ to its $a_i$ leaves in counterclockwise order. Let $G_m$
be a copy of $G$ in $\K_n$, and let $\ell_m(e):=\ell(F_{G,G_m}(e))$; that is,
$\ell_m(e)$ is the length of the edge of $G_m$ corresponding to the
edge $e$ of $G$. We will show, for all $m\in\{1,\ldots, x-k^2-2\}$, that there
exists $G_m$ with
$$
\arraycolsep=1.4pt
\begin{array}{rlll}
\ell_m(e_i)=&(\sigma_i-1)x+m, & &~~~~~i\in\{1,\ldots,s\}\\
\ell_m(e_{i,t})=&(s+\sum_{j=1}^{i-1}a_j)x+i\cdot k+m\cdot a_i+t,& &
~~~~~i\in\{1,\ldots, s+1\}\mbox{ and }t\in\{1,\ldots,a_i\}.
\end{array}
$$
First, we verify that each edge length $\ell_m$ is less than $n/2$. Since $\sigma_i\leq s$ and $m < x$, we have $\ell_m(e_i)< \sigma_ix\le sx \le k(n/(2k))= n/2$; and for all $a_i \neq 0$ we have $$\ell_m(e_{i,t})\leq (s+\sum_{j=1}^{i-1}a_j)x+k^2+(x-k^2-1) a_i \leq e(G)x-1< \frac{n}{2}.$$ 
We need to check that edges $e_{i,a_i}$ and $e_{i+1,1}$ do not cross, for all
$i\in\{1,\ldots,s\}$ in each copy of $G$. It suffices to check that
$\ell_m(e_{i,a_i})+\ell_m(e_i)\le \ell_m(e_{i+1,1})$ for all $i\in\{1,\ldots,
s\}$ and all $m\in\{1,\ldots, x-k^2-2\}$. Note that
\begin{align*}
\ell_m(e_{i+1,1})-\ell_m(e_{i,a_i})-\ell_m(e_i)
=&(s+\sum_{j=1}^{i}a_j)x+(i+1)k+m a_{i+1}+1\\
&-(s+\sum_{j=1}^{i-1}a_j)x-i k-(m+1) a_i-(\sigma_i-1)x-m\\
=&(a_i-\sigma_i+1)x+(a_{i+1}-a_i-1)m+k-a_{i}+1\\ 
\end{align*}
If $a_{i+1}\ge a_i$, then (since $a_i\ge \sigma_i$, by assumption) we have
\begin{align*}
\ell_m(e_{i+1,1})-\ell_m(e_{i,a_i})-\ell_m(e_i)
& \ge x-m+k-a_{i}+1\ge k-a_{i}+1>0.
\end{align*}
On the other hand, if $a_{i+1}< a_i$, then (since $m\le x$) we have
\begin{align*}
\ell_m(e_{i+1,1})-\ell_m(e_{i,a_i})-\ell_m(e_i)
\ge&(a_i - a_{i+1}+1)(x-m)+k-a_{i}+1\ge k-a_{i}+1>0.
\end{align*}
Hence such $G_m$ exist for all $m\in\{1\ldots x-k^2\}$. Let $S_n$ be the collection of
sets consisting of $L_{G_m}$ for all $m\in\{1,\ldots x-k^2-2\}$. It is easy to
check that sets in $S_n$ are pairwise disjoint and
$\sum_{m=1}^x|L_{G_m}|=(1-o(1))n/2$. Therefore, $G$ is strongly packable.
\end{proof}

\section{Concluding Remarks}
\label{sec:conclude}
To conclude, we discuss some problems that remain open. Here we propose three directions for future research.

\begin{enumerate}
    \item Let $G$ be a graph. If $\P(G)$ is geometric-packable, then clearly $\P^*(G)$ can be asymptotically packed into convex complete graphs. Is the converse true? In particular, can we pack into any geometric drawings of the complete graph the sets $\P(\Theta_2)$, $\P(\Theta_3)$ or $\P(\Theta_4)$ (which are convex-packable by Theorem~\ref{thm:CyclesWithChords})?
    \item Lemma~\ref{lem:notpackability} is currently our only tool to prove that a CGG is not convex-packable. It would be enlightening to discover an example that fails to satisfy the hypotheses of Lemma~\ref{lem:notpackability} but is still not convex-packable.  Does such a graph exist?
    \item We ask for the maximum number of extremal edges $f(k)$ of a convex-packable plane path with $k$ edges. We have shown that $(1+o(1))\sqrt{2k}\le f(k)\le 2\sqrt{k}$ and we conjecture that the upper bound is true. We can also ask for the minimum number of extremal edges $g(k)$ of a convex-nonpackable plane path with $k$ edges (if there is one).  For the upper bound, we know that $g(k)\le 2\sqrt{k}$. On the other hand, it seems to be non-trivial even to find a lower bound that grows unbounded as a function of $k$.
\end{enumerate}
\section*{Acknowledgments}
Most research in this paper took place at the 2021 Graduate Research Workshop
in Combinatorics.  We heartily thank the organizers.  
We also thank 
Abdul Basit,
Austin Eide,
Bernard Lidick\'{y}, 
and
Shira Zerbib 
for helpful discussions on this problem.
\bibliographystyle{abbrv}

{\footnotesize{\bibliography{refs}}}
\end{document}